\documentclass{amsart}
\usepackage{amssymb,amsmath,color}
\usepackage{color}
\usepackage{tikz}
\usepackage{pgf}
\usepackage{bm}
 \usepackage[left=1in,right=1in,top=0.9in,bottom=0.9in]{geometry}

\newtheorem{theorem}{Theorem}[section]
\newtheorem{conjecture}[theorem]{Conjecture}
\newtheorem{lemma}[theorem]{Lemma}
\newtheorem{corollary}[theorem]{Corollary}

\theoremstyle{definition}

\theoremstyle{remark}

\def \Hom {\operatorname{Hom}}
\def \HDE {\operatorname{HDE}}

\date{\today}

\title{Proof of the Erd\H{o}s-Simonovits conjecture on walks}
\author{Grigoriy Blekherman}
\address{School of Mathematics, Georgia Institute of Technology,
686 Cherry Street
Atlanta, GA 30332}\email{greg@math.gatech.edu}

\author{Annie Raymond}
\address{Department of Mathematics and Statistics,
Lederle Graduate Research Tower, 1623D,
University of Massachusetts Amherst
710 N. Pleasant Street
Amherst, MA 01003} \email{raymond@math.umass.edu}

\thanks{The authors are indebted to Mike Tait who shared this conjecture with them. Grigoriy Blekherman was partially supported by NSF grant DMS-1901950.}
\begin{document}

\begin{abstract}
Let $G^n$ be a graph on $n$ vertices and let $w_k(G^n)$ denote the number of walks of length $k$ in $G^n$ divided by $n$. Erd\H{o}s and Simonovits conjectured that $w_k(G^n)^t \geq w_t(G^n)^k$ when $k\geq t$ and both $t$ and $k$ are odd. We prove this conjecture.
\end{abstract}

\maketitle

\section{Introduction}

Let $G^n$ be a graph on $n$ vertices, let $e(G^n)$ be the number of edges in $G^n$, and let $w_k(G^n)$ denote the number of walks of length $k$ (i.e., with $k$ edges) in $G^n$ divided by $n$. In \cite{erdossimonovits}, Conjecture 6 reads as follows:

\begin{conjecture}[Erd\H{o}s-Simonovits, 1982]\label{conj:es}
If $d$ is the average degree in $G^n$, i.e., $d=\frac{2e(G^n)}{n}$ then 
$$w_k(G^n) \geq d^k,$$ further if $k \geq t$, and both $t$ and $k$ are odd, then
$$w_k(G^n)^t \geq w_t(G^n)^k.$$
\end{conjecture}

Erd\H{o}s and Simonovits mention that the first inequality in the conjecture had already been proven in \cite{blakleyroy}, \cite{london} and \cite{mulhollandsmith}. Today, it is known as the Blakley-Roy inequality. They then go on to remark that the second inequality in the conjecture is a generalization of the first, and it is known to hold when $k$ is even and they give a proof due to C.D. Godsil. The authors finally point out that the second inequality does not hold when $k$ is odd and $t$ is even.

In this paper, we prove the remaining case of the conjecture: we prove the second inequality when both $t$ and $k$ are odd. To do so, we reformulate the conjecture in terms of numbers of graph homomorphisms, and then apply a theorem from Kopparty and Rossman \cite{koppartyrossman}. We present the background needed in Section \ref{background}, and the proof in Section \ref{proof}.

\section{Reformulation and a theorem by Kopparty and Rossman}\label{background}

Let $V(G)$ and $E(G)$ denote respectively the vertex set and the edge set of a graph $G$. A $\emph{graph homomorphism}$ from a graph $F$ to a graph $G$ is a map from the vertex set of $F$ to the vertex set of $G$ that sends edges to edges, i.e., that preserves adjacency. More precisely, a graph homomorphism is a function $\varphi : V(F) \rightarrow V(G)$ such that for any edge $\{v_1, v_2\}\in E(F)$, $\{\varphi(v_1), \varphi(v_2)\}\in E(G)$. Let $\Hom(F;G)$ be the set of homomorphisms from $F$ to $G$. Let $t(F;G)$ be the homomorphism density of $F$ in $G$, i.e., the probability that a random map from the $V(F)$ to $V(G)$ is a graph homomorphism. Note that $t(F;G)=\frac{|\Hom(F;G)|}{|V(G)|^{|V(F)|}}$. One well-known property is that $|\Hom(F_1;G)|\cdot |\Hom(F_2;G)| = |\Hom(F_1F_2;G)|$ and $t(F_1;G)\cdot t(F_2;G) = t(F_1F_2;G)$ where $F_1F_2$ denotes the disjoint union of $F_1$ and $F_2$. 

In this paper, $G$ will normally vary over all graphs on $n$ vertices. To lighten the notation, in inequalities, we will write $F$ to mean the function that can be evaluated on graphs $G$ by taking the number of homomorphisms from $F$ to $G$. The property $|\Hom(F_1;G)|\cdot |\Hom(F_2;G)| = |\Hom(F_1F_2;G)|$ thus becomes $F_1 \cdot F_2 = F_1F_2$.

Let $P_k$ be the function that evaluates the number of homomorphisms from a path with $k$ edges to some graph $G$ on $n$ vertices. Note that $\frac{P_k}{n} = w_k(G)$. When $k=0$, $P_0$ is a single vertex (i.e., a 0-path), and thus $P_0=n$.  The second part of Conjecture 6 from \cite{erdossimonovits} can thus be reformulated as $$\left(\frac{P_k}{n}\right)^t \geq \left(\frac{P_t}{n}\right)^k$$ when $k \geq t$ and both $t$ and $k$ are odd. Another way to formulate the conjecture is to say that $n^{k-t} {P_k}^t \geq {P_t}^k$ or that $P_0^{k-t} P_k^t \geq P_t^k$  for all $k\geq t$ where $t$ and $k$ are both odd. Finally, observe that by dividing $n^{k-t} P_k^t \geq P_t^k$ by $n^{k(t+1)}$ on both sides, we obtain $t(P_k^t; G^n) \geq t(P_t^k; G^n)$ or equivalently $t(P_k; G^n)^t \geq t(P_t; G^n)^k$, which thus yields another way of formulating the conjecture.

\begin{lemma}\label{lem:suffices}
To prove the conjecture, it suffices to show that $P_0^2 P_{t+2}^{t} \geq P_{t}^{t+2}$ for any odd $t+2$. 
\end{lemma}

\begin{proof}
Suppose that we know that $P_0^2 P_{t+2}^{t} \geq P_{t}^{t+2}$ for any odd $t+2$. This is equivalent to knowing that $t(P_{t+2}; G^n) \geq t(P_{t}; G^n)^{\frac{t+2}{t}}$. If $k= t+2i$ where $i>1$, then we have

\begin{align*}
t(P_k;G^n) & = t(P_{t+2i}; G^n) \\
& \geq t(P_{t+2i-2}; G^n)^{\frac{t+2i}{t+2i-2}} \\
& \geq \left(t(P_{t+2i-4};G^n)^{\frac{t-2i-2}{t-2i-4}}\right)^{\frac{t+2i}{t+2i-2}}\\
& \geq \ldots \\
& \geq \left(\left(\left(t(P_{t};G^n)^{\frac{t+2}{t}}\right)^{\frac{t+4}{t+2}} \ldots  \right)^{\frac{t-2i-2}{t-2i-4}}\right)^{\frac{t+2i}{t+2i-2}}\\
& = t(P_{t};G^n)^{\frac{t+2i}{t}} = t(P_{t};G^n)^{\frac{k}{t}}
\end{align*}
as desired. 
\end{proof}

The concept of homomorphism domination exponent was introduced in \cite{koppartyrossman}, though the idea behind it had been central to many problems in extremal graph theory for a long time. Let the \emph{homomorphism domination exponent} of a pair of graphs $F_1$ and $F_2$, denoted by $\HDE(F_1;F_2)$, be the maximum value of $c$ such that $|\Hom(F_1;G)| \geq |\Hom(F_2;G)|^c$ for every graph $G$. Thus, by Lemma \ref{lem:suffices}, to prove the conjecture, it suffices to show that $\HDE(P_0^2 P_{t+2}^{t}; P_{t})=t+2$ for any odd $t$ (where we now think simply of $P_i$ as a graph, namely the path with $i$ edges, and not as a function). 

In \cite{koppartyrossman}, Kopparty and Rossman showed that $\HDE(F_1; F_2)$ can be found by solving a linear program when $F_1$ is chordal and $F_2$ is series-parallel. Since this is the case when $F_1=P_0^2 P_{t+2}^t$ and $F_2=P_{t}$, we will use this linear program to prove the conjecture. We now briefly describe Kopparty and Rossman's result which is based on comparing the entropies of different distributions on $\Hom(F_2; G)$. We later pull back such distributions, and in particular the uniform distribution on all homomorphisms.

Let $\mathcal{P}(F_2)$ be the polytope consisting of normalized $F_2$-polymatroidal functions, which is defined to be
\begin{small}
\begin{align*}
\mathcal{P}(F_2)=\big\{p\in \mathbb{R}^{2^{|V(F_2)|}} | & p(\emptyset) = 0 &&\\
& p(V(F_2)) = 1 && \\
& p(A) \leq p(B) && \forall \  A \subseteq B \subseteq V(F_2)\\
& p(A\cap B) + p(A\cup B) \leq p(A)+p(B) && \forall \ A, B \subseteq V(F_2)\\
& p(A\cap B) + p(A\cup B) = p(A)+p(B) &&\forall \ A, B \subseteq V(F_2) \textup { such that } A\cap B \\
&&& \quad \quad \quad \textup{ separates } A\backslash B \textup{ and }B\backslash A\big\}.
\end{align*}
\end{small}
Note that $A\cap B$ is said to separate $A\backslash B$ and $B\backslash A$ if there are no edges in $F_2$ between $A\backslash B$ and $B\backslash A$. 

\begin{theorem}[Kopparty-Rossman, 2011]
Let $F_1$ be a chordal graph and let $F_2$ be a series-parallel graph. Then

$$\HDE(F_1, F_2) = \min_{p\in \mathcal{P}(F_2)} \max_{\varphi\in \Hom(F_1;F_2)}  \sum_{S \subseteq \textup{MaxCliques}(F_1)} -(-1)^{|S|} p(\varphi(\cap S))\\$$
where $\textup{MaxCliques}(F_1)$ is the set of maximal cliques of $F_1$ and $\cap S$ is the intersection of the maximal cliques in $S$. 
\end{theorem}

\section{Proof of the conjecture}\label{proof}

Let $[m]:=\{1,2,\ldots, m\}$,  $V(P_t)=[t+1]$, and let $E(P_t)=\{\{i,i+1\}| i \in [t]\}$. Lemma 2.5 of \cite{koppartyrossman} implies that for any $p\in \mathcal{P}(P_t)$, $$p(V(P_t))=\sum_{S\subseteq\textup{MaxCliques}(P_t)} -(-1)^{|S|}p(\cap S).$$ For completeness, we give a short argument.

\begin{lemma}\label{pall}
For any $p\in \mathcal{P}(P_t)$ for some $t\geq 1$ (not necessarily odd), $$p(V(P_t))=\sum_{\{i, i+1\}\in E(P_t)}p(\{i,i+1\}) -\sum_{i \in \{2,\ldots,t\}}p(\{i\}).$$
\end{lemma}

\begin{proof}
We prove it by induction on $t$. If $t=1$, it is trivially true since there are no negative terms to consider. Suppose it is true for $t$. Consider $P_{t+1}$. Let $A=[t+1]$ and let $B=\{t+1, t+2\}$. Then $A\cup B= [t+2]=V(P_{t+1})$, and $A\cap B= \{t+1\}$. Note that $A\cap B$ separates $A\backslash B$ and $B\backslash A$, so $p(A\cup B)= p(A)+p(B)-p(A\cap B)$. Thus
\begin{align*}
p(V(P_{t+1}))&=p(V(P_t))+p(\{t+1, t+2\})-p(\{t+1\})\\
&=p(\{1,2\})+\ldots + p(\{t,t+1\})-p(\{2\}) - \ldots -p(\{t\}) +p(\{t+1, t+2\})-p(\{t+1\})\\
&=\sum_{\{i, i+1\}\in E(P_{t+1})}p(\{i,i+1\}) -\sum_{i \in \{2,\ldots,t+1\}}p(\{i\}),
\end{align*}
where the second line follows from the induction hypothesis.
\end{proof}

\begin{theorem}\label{maintheorem}
We have that $\HDE(P_0^2 P_{t+2}^t; P_t)=t+2$, and thus that Conjecture \ref{conj:es} holds. 
\end{theorem}

\begin{proof}
We first show that $\HDE(P_0^2 P_{t+2}^t; P_t)\leq t+2$. For $i \in [t+1]$ and $S\subseteq [t+1]$, let $p_i\in \mathbb{R}^{2^{t+1}}$ be such that $p_i(S)=1$ if $S$ contains $i$, and $p_i(S)=0$ otherwise. It's easy to check that $p_i \in \mathcal{P}(P_t)$. Let $p^*$ be the average of the $p_i$'s, i.e., $p^*= \frac{1}{t+1} \sum_{i\in [t+1]} p_i$. In particular, this means that $p^*(\{i\})= \frac{1}{t+1}$ for any $i\in [t+1]$, and $p^*(\{i,i+1\})=\frac{2}{t+1}$ for any $i\in [t]$. Since $p^*$ is a convex combination of the $p_i$'s, $p^*\in \mathcal{P}(P_t)$. For any homomorphism $\varphi$ from $P_0^2 P_{t+2}^t$ to $P_t$,

$$\sum_{S \subseteq \textup{MaxCliques}(P_0^2 P_{t+2}^t)} -(-1)^{|S|} p^*(\varphi(\cap S)) = t\cdot(t+2)\frac{2}{t+1} - t\cdot (t+1)\frac{1}{t+1} + 2\frac{1}{t+1} = t+2,$$ which implies that the optimal value of the linear program is at most $t+2$.

We now show that $\HDE(P_0^2 P_{t+2}^t; P_t)\geq t+2$. For $1\leq i \leq t$ let $\varphi_{i}$ be the homomorphism from $P_{t+2}$ to $P_t$ such that $\varphi_{i}(j)=j$ for all $j\leq i+1$, and $\varphi_{i}(j)=j-2$ for all $j\geq i+2$. In other words, every edge of $P_t$ is visited by $P_{t+2}$ once, except for $\{i, i+1\}$ which is visited three times. Let $\psi$ be the homomorphism from $P_0^2 P_{t+2}^t$ to $P_t$ such that one copy of $P_0$ gets sent to vertex $1$ in $P_t$, the other copy of $P_0$ is sent to vertex $t+1$ of $P_t$ (i.e., the two copies of $P_0$ are sent to the end vertices of $P_t$), and the $i$-th copy of $P_{t+2}$ is mapped to $P_t$ via $\varphi_i$ for $1\leq i\leq t$.

Now for any $p \in \mathcal{P}(P_t)$, we compute $$\sum_{S \subseteq \textup{MaxCliques}(P_0^2 P_{t+2}^t)} -(-1)^{|S|} p(\psi(\cap S)).$$ Observe that only sets $S$ of size one or two contribute in the above sum since no three maximal cliques of $P_0^2 P_{t+2}^t$ intersect. Every edge of $P_t$ is covered by an image of an edge of $P_{t+2}$ via $\psi$ exactly $t+2$ times. Every inner (non-end) vertex of $P_t$ is covered by an image of an inner (non-end) vertex of $P_{t+2}$ via $\psi$ exactly $t+2$ times. Note that each inner vertex of some copy of $P_{t+2}$  is the intersection of two maximal cliques (i.e., edges) of $P_{t+2}$, and thus the coefficient will be negative. Finally, the end vertices of $P_t$ are covered by an image of an inner (non-end) vertex of $P_{t+2}$ via $\psi$ exactly once each (which brings again a negative coefficient as it is the intersection of two maximal cliques), as well as once each by one copy of $P_0$ (which brings a positive coefficient as each $P_0$ is a maximal clique in itself). Thus the coefficients for the end vertices of $P_t$ are zero. Accordingly we have  
\begin{align*}
\sum_{S \subseteq \textup{MaxCliques}(P_0^2 P_{t+2}^t)} -(-1)^{|S|} p(\psi(\cap S))&=(t+2)\left(\sum_{\{i, i+1\}\in E(P_t)}p(\{i,i+1\}) -\sum_{i \in \{2,\ldots,t\}}p(\{i\})\right)\\
&=(t+2)p(V(P_t))\\
&=t+2.
\end{align*}

The second line follows from Lemma \ref{pall}, and the third line follows from $p(V(P_t))=1$ since $p\in \mathcal{P}(P_t)$. Therefore, for every $p \in \mathcal{P}(P_t)$, there is an homomorphism that yields $t+2$, so we see that $\HDE(P_0^2 P_{t+2}^t; P_t)\geq t+2$. This proves that $\HDE(P_0^2 P_{t+2}^t; P_t)= t+2$, and therefore  the conjecture holds.
\end{proof}

\begin{corollary}
We also have that $t(P_k; G^n)^t \geq t(P_t; G^n)^k$ holds.
\end{corollary}

\bibliographystyle{alpha}
\bibliography{esreferences}

 \end{document}